\documentclass[10pt,reqno]{amsart}
\usepackage{amsfonts,amsmath,amssymb,amsbsy,amstext,amsthm}
\usepackage{accents,color,enumerate,enumitem,float,fullpage,parskip,verbatim}

\usepackage{url}
\usepackage[colorlinks=true,hyperindex, linkcolor=magenta, pagebackref=false, citecolor=cyan]{hyperref}
\usepackage[alphabetic,lite,backrefs]{amsrefs} 

\usepackage{eucal,bm,kpfonts,mathbbol}

\usepackage{tikz,tikz-cd}	
\usetikzlibrary{positioning, matrix, shapes}         								    				

\usepackage{lscape}

\usepackage{titlesec}		
\titleformat{\section}[block]{\large\scshape\bfseries\filcenter}{\thesection.}{1em}{}		
\titleformat{\subsection}[hang]{\large\scshape\bfseries}{\thesubsection}{1em}{}			
\titleformat{\subsubsection}[hang]{\large\scshape\bfseries}{\thesubsubsection}{1em}{}			

\usepackage[titles]{tocloft}								     					
\usepackage{multirow}
\usepackage{array}
\usepackage{booktabs}
\newcolumntype{M}[1]{>{\centering\arraybackslash}m{#1}}
\newcolumntype{N}{@{}m{0pt}@{}}
\usepackage{diagbox}
\usepackage{cancel}

\newtheorem{lemma}{Lemma}[section]
\newtheorem{theorem}[lemma]{Theorem}
\newtheorem{prop}[lemma]{Proposition}
\newtheorem{cor}[lemma]{Corollary}

\newtheorem*{thmA}{Theorem A}
\newtheorem*{thmB}{Theorem B}

\theoremstyle{remark}
\newtheorem{remark}[lemma]{Remark}

\newcommand{\reg}{\operatorname{reg}}

\newcommand{\Spec}{\operatorname{Spec}}

\newcommand{\Hom}{\operatorname{Hom}} 

\newcommand{\Jac}{\operatorname{Jac}}

\newcommand{\red}{\operatorname{red}}
\newcommand{\proj}{\operatorname{proj}}

\newcommand{\ff}{\mathbf f}
\newcommand{\kk}{\mathbf k}

\newcommand{\dd}{\mathbf d}

\newcommand{\fF}{\mathbf F}

\newcommand{\cA}{\mathcal{A}}

\newcommand{\cM}{\mathcal{M}}

\newcommand{\cP}{\mathcal{P}}

\newcommand{\cT}{\mathcal{T}}

\newcommand{\N}{\mathbb{N}}
\renewcommand{\P}{\mathbb{P}}

\newcommand{\V}{\mathbb{V}}

\newcommand{\Z}{\mathbb{Z}}

\title{Effective Bounds on the Dimensions of Jacobians Covering Abelian Varieties }
\author{Juliette Bruce}
\address{Department of Mathematics, University of Wisconsin, Madison, WI}
\email{\href{mailto:juliette.bruce@math.wisc.edu}{juliette.bruce@math.wisc.edu}}
\urladdr{\url{http://math.wisc.edu/~juliettebruce/}}

\author{Wanlin Li}
\address{Department of Mathematics, University of Wisconsin, Madison, WI}
\email{\href{mailto:wanlin@math.wisc.edu}{wanlin@math.wisc.edu}}
\urladdr{\url{http://math.wisc.edu/~wanlin/}}

\begin{document} 
\thanks{The first author was partially supported by the NSF GRFP under Grant No. DGE-1256259.}

\maketitle

\begin{abstract}
We show that any abelian variety over a finite field is covered by a Jacobian whose dimension is bounded by an explicit constant. We do this by first proving an effective and explicit version of Poonen's Bertini theorem over finite fields, which allows us to show the existence of smooth curves arising as hypersurface sections of bounded degree and genus. Additionally, for simple abelian varieties we prove a better bound. As an application, we show that for any elliptic curve $E$ over a finite field and any $n\in \N$, there exist smooth curves of bounded genus whose Jacobians have a factor isogenous to $E^n$.
\end{abstract}

\setcounter{section}{1}
Over an infinite field, every abelian variety is covered by the Jacobian variety of a smooth connected curve. In fact, given an embedding of the abelian variety, one can even provide an effective upper bound on the dimension of the Jacobian variety using the dimension and degree of the abelian variety (see \cite[Section III]{milneAV}). We show that an analogous effective statement holds over a finite field.

\begin{thmA}\label{thm:main-theorem}
Fix $r,n\in \N$ with $n\geq2$, and let $\fF_{q}$ be a finite field of characteristic $p$. There exists an explicit constant\footnote{See Proposition~\ref{cor:degree-bound} for a more precise statement where the constant is explicitly stated.}  $C_{r,q}$ such that if $A\subset \P^{r}_{\fF_q}$ is a non-degenerate abelian variety of dimension $n$, then for any $d\in \N$ satisfying 
\[
C_{r,q}\zeta_{A}\left(n+\tfrac{1}{2}\right) \deg(A) \leq  \frac{q^{\frac{d}{\max\{n+1,p\}}}\left(d+1\right)}{d^{n+1}+d^n+q^{\frac{d}{\max\{n+1,p\}}}},
\]
there exists a smooth geometrically connected curve over $\fF_{q}$ whose Jacobian $J$ maps dominantly onto $A$, where 
\[
\dim J\leq 
\left\lfloor \frac{\deg(A)d^{n-1}-1}{r-1}\right\rfloor \left(\deg(A)d^{n-1}-\frac{\left\lfloor\frac{\deg(A)d^{n-1}-1}{r-1}\right\rfloor+1}{2}(r-1)-1\right).
\]
Moreover, if $A\subset \P^{r}_{\fF_{q}}$ is simple, then for any $d\in \N$ satisfying
\[
\deg(A) \leq \frac{(d-1)q^{\tfrac{1}{2}(d+1)(d+2)}}{d^{n-1}-1},
\]
there exists a smooth geometrically connected curve over $\fF_{q}$ whose Jacobian $J$ maps dominantly onto $A$, where 
\[
\dim J\leq 
\deg(A)d^{n-1}\left(\deg(A)d^{n-1}+1\right).
\]
\end{thmA}

Over an infinite field the fact that every abelian variety is covered by the Jacobian variety of a smooth connected curve is long known. The key idea, which we review (and slightly extend) in Proposition~\ref{prop:dominate-map-av}, is this: if $A\subset \P^r_{\kk}$ is an embedded $n$-dimensional abelian variety, and $C$ is a smooth curve which arises as the intersection of $A$ with a linear subspace $L\subset \P^r_{\kk}$ of codimension $n-1$, then $\Jac(C)$ will cover $A$. It is thus sufficient to find a linear subspace of codimension $n-1$ which intersects $A$ in a smooth curve. Over an infinite field, such a linear space exists by Bertini's Theorem.

When the base field $\kk$ is a finite field, the situation is substantially more subtle. For instance, it need no longer be the case that there exists even a single hyperplane in $\P^r_{\kk}$ that has a smooth intersection with $A$. Poonen's Bertini Theorem shows that while one cannot necessarily find smooth hyperplane sections, smooth hypersurface sections always exist if the degree of the hypersurface is allowed to be arbitrarily high \cite[Theorem~1.1]{poonen}. By induction, there exist homogeneous polynomials $f_1,\ldots,f_{n-1}$ of high enough degree such that $A\cap \V(f_1,\ldots,f_{n-1})$ is a smooth connected curve. This implies the existence of a Jacobian variety mapping dominantly onto $A$ when $\kk$ is a finite field. 

While Poonen's result is enough to show existence, it is not enough to provide the explicit bounds appearing in Theorem~A. For example, one does not necessarily know what the degrees of $f_{1},\ldots,f_{n-1}$ may be. In fact, since the construction of the $f_k$ is inductive, it may be the case that the choice of $f_{1},\ldots,f_{k-1}$ affects the degree of $f_{k}$. Existence was also proved over finite fields independently by Gabber using different methods \cite[Corollary~2.5]{gabber01}; however, this also does not provide explicit bounds. 

We prove Theorem~A by first proving an effective version of Poonen's result with explicit bounds. 

\begin{thmB}
Fix $r,n\in \N$ with $n\geq2$, and let $\fF_{q}$ be a finite field of characteristic $p$. For any  $1\leq k \leq n-1$ there exists an explicit constant\footnote{See Proposition~\ref{prop:effective-bucur-kedlaya} for a more precise statement where the constant is explicitly stated.}  $C_{r,q}$ such that if $X\subset \P^{r}_{\fF_{q}}$ is a smooth quasi-projective subscheme of dimension $n$, then for any $d\in \N$ satisfying
\[
C_{r,q}\deg(X)\zeta_{X}\left(n+\tfrac{1}{2}\right)<\frac{q^{\frac{d}{\max\{n+1,p\}}}\left(d^{\frac{2k-1}{n}}+1\right)}{d^{n}+d^{n+\frac{2k-1}{n}}+q^{\frac{d}{\max\{n+1,p\}}}},
\]
there exist homogeneous polynomials $f_1, \ldots, f_k \in \fF_{q}[x_0,\ldots,x_r]$ of degree $d$ such that $X\cap \V(f_1, \ldots, f_k)$ is smooth of dimension $n-k$. Moreover, if $X$ is projective and geometrically connected then $X\cap \V(f_1, \ldots, f_k)$ is also geometrically connected.
\end{thmB}

Our proof of this theorem builds upon work of Bucur and Kedlaya \cite{bucurKedlaya12}, which in part allows us to choose all of the hypersurfaces at once instead of going through an inductive argument. We prove non-trivial bounds on the error terms and the Euler product appearing in \cite[Theorem~1.2]{bucurKedlaya12}, which allow us to deduce the explicit bound appearing in Theorem~B.

From this effective Bertini theorem over finite fields applied to abelian varieties, we deduce Theorem~A as follows: first we use Theorem~B to produce a smooth connected curve $C$ on $A$ whose degree is explicitly bounded and which arises as an intersection $C= A \cap V(f_1, \ldots, f_{n-1})$. Then we use Proposition \ref{prop:dominate-map-av} to show that $\Jac (C)$ covers $A$, and finally we use a classical theorem of Castelnuovo to bound the genus of $C$. 

In the case when the abelian variety $A$ is simple, the condition of $A\cap \V(f_1,\ldots,f_{n-1})$ being smooth can be dropped, and this allows us to lower the degree and genus bounds. To construct an explicit smooth curve whose Jacobian dominates $A$, we just need a curve (not necessarily smooth or even reduced) given by the intersection of $A$ with hypersurfaces. Using recent work of the first author and Erman, characterizing the probability of randomly choosing homogeneous polynomials $f_{1},\ldots,f_{n-1}$ of degree $d$ that intersect $A$ in a (not necessarily smooth) curve \cite[Theorem~B, Proposition~5.1]{bruce16}, we show that when $A$ is simple, hypersurfaces of smaller degree suffice. This results in the better bound seen in Theorem~A.

Since we work with non-smooth curves in the case where $A$ is simple, we cannot use Castelnuovo's bound for the genus. We thus prove a more general degree-genus bound that holds for any connected, reduced curve. The key idea of this proof is to combine a Hilbert function argument with the Gruson-Lazarsfeld-Peskine bound on Castelnuovo-Mumford regularity of any such curve \cite{gruson83, giaimo06}. 

As an application of Theorem~A, we show the existence of a smooth connected curve with bounded genus whose Jacobian has an arbitrary number of copies of an elliptic curve as isogeny factors.

\begin{cor}\label{mainapp}
Let $\fF_q$ be a finite field and for any $n\in \N$, there exists an explicit constant $B_{n,q}$ such that for any $E$, an elliptic curve over $\fF_{q}$, there exists a smooth geometrically connected curve $C$ of genus $g \leq B_{n,q}$ defined over $\fF_{q}$ such that $\Jac(C)$ admits $E^n$ as an isogeny factor.
\end{cor}

This paper is organized as follows. \S\ref{sec:facts-abelian} gathers background results about abelian varieties.  In \S\ref{sec:prob-curves} we prove Theorem~B. In \S\ref{sec:bound-genus} we use Theorem~B to prove the general statement in Theorem~A. \S\ref{sec:simple-case} concludes the proof of Theorem~A by handling the case of simple abelian varieties. \S\ref{sec:applications} presents the proof of Corollary~\ref{mainapp}.


\section*{Acknowledgements}  We thank Alina Bucur, Brian Conrad, Daniel Erman, Jordan S.\ Ellenberg, Mois\'es Herrad\'on Cueto, Kiran S.\ Kedlaya, Kit Newton, Solly Parenti, Bjorn Poonen, and David Zureick-Brown for their helpful conversations and comments. We would like to thank the referees for many helpful comments.


\section*{Conventions} We let $\N=\{1,2,3,\ldots\}$ be the natural numbers and $\Z$ be the integers. Throughout the paper, $\kk$ will denote a field, and $\fF_{q}$ will be a finite field of characteristic $p$ for some prime $p>0$. By a curve over a field $\kk$, we refer to a complete separated equidimensional scheme of finite type over $\kk$ of dimension one. By equidimensional we mean that all of the irreducible components have the same dimension and that there are no embedded components. We will say a scheme $X$ over a field $\kk$ is smooth if its structure morphism is smooth. We discuss the Jacobian variety associated to a smooth connected curve $C$ as defined in \cite[Section~III.1, pg. 86]{milneAV} and we denote the Jacobian variety of such a curve $C$ by $\Jac(C)$. By abelian variety over a field $\kk$, we mean a geometrically reduced, separated, group scheme of finite type over $\kk$ that is both complete and geometrically connected \cite[Section~I.1, pg. 8]{milneAV}. When discussing a polynomial ring $\kk[x_0,\ldots,x_r]$ over a field $\kk$, we will always assume it has the standard $\N$-grading where $\deg(x_i)=1$ for all $i$. 


\section{Background on Abelian Varieties}\label{sec:facts-abelian}

Here we collect some classical results regarding abelian varieties, each adapted from \cite[Section~III]{milneAV}.

Let $A$ be an abelian variety and $C$ be a smooth connected curve together with a map $C \rightarrow A$. By the universal property of Jacobians, one has the following diagram:
\[
\begin{tikzcd}[row sep = 3em, column sep = 3em]
C \rar{\iota} \arrow[dr]& \Jac(C) \arrow[d,dashed, "\pi"]\\
& A
\end{tikzcd}
\]
where $\iota:C\rightarrow \Jac(C)$ is an Abel-Jacobi map for $C$. In general, the map $\pi$ need not be surjective. However, if the curve $C$ arises as a complete intersection on $A$ -- i.e. if there exist homogeneous polynomials $f_1,\ldots,f_{n-1}$ on $\P^{r}_{\kk}$ such that $C=A\cap \V(f_{1},\ldots,f_{n-1})$ -- then the map $\pi$ is surjective. This is the content of the following proposition.

\begin{prop}\label{prop:dominate-map-av}
Let $A \subset \P^{r}_{\kk}$ be an abelian variety of dimension $n$ over a field $\kk$. If $f_1,\ldots,f_{n-1}\in \kk[x_0,\ldots,x_r]$ are homogeneous polynomials such that $C\coloneqq A\cap\V(f_{1},\ldots,f_{n-1})$ is a smooth geometrically connected curve, then the induced map $\pi:\Jac(C)\rightarrow A$ is surjective.
\end{prop}

The case where the $f_i$'s are linear forms is Theorem 10.1 in \cite[Section III]{milneAV}, and the key adaptation here is allowing hypersurfaces of higher degree. To do this, we need the following lemma, which is adapted from Lemma 10.3 in \cite[Section III]{milneAV}.

\begin{lemma}\label{lemma:hyersurface-finite-maps}
Let $X\subset \P^{r}_{\kk}$ be a projective subscheme of dimension $\geq 2$ defined over a field $\kk$, and $X'=X\cap H$ be a hypersurface section of $X$. Let $Y$ be a geometrically normal, geometrically integral, projective scheme. If $\psi:Y\rightarrow{}X$  is a finite map then $\psi^{-1}(X')$ is geometrically connected. 
\end{lemma}

\begin{proof}
Since the hypotheses are stable under base change, it is enough to assume that $\kk$ is algebraically closed and show that $\psi^{-1}(X')$ is connected. Since $X'$ is the restriction of an ample divisor on $\P^{r}_{\kk}$ to $X$, it remains ample. Since $\psi$ is a finite morphism, it is quasi-affine. Thus $\psi^{-1}(X')$ is the support of an ample divisor \cite[\href{https://stacks.math.columbia.edu/tag/0892}{Lemma 0892}]{stacks-project}. Finally, since $\kk$ is algebraically closed and $Y$ is integral and normal we may apply Corollary 7.9 of \cite[Section III]{hartshorne}, which implies the desired claim. 
\end{proof}

\begin{proof}[Proof of Proposition~\ref{prop:dominate-map-av}]
 Consider the image of $\pi$, which we denote by $A_1$. Compositing with a translation on $A$ wouldn't affect surjectivity of the map. Without loss of generality, we assume $\pi$ is a group homomorphism. Thus, $A_1$ is an abelian subvariety of $A$. Towards a contradiction, suppose that $A_1\neq A$. If this was the case, then there exists an abelian subvariety $A_2\subset A$ such that the map $\phi : A_1 \times A_2 \rightarrow A$ given by $(a_1,a_2)\mapsto a_1+a_2$ is an isogeny \cite[Proposition I.10.1]{milneAV}. 

Now let $m\in \N$ be relatively prime to the characteristic of $\kk$, and consider the map $\psi$:
\[
\begin{tikzcd}[column sep = 3em]
 A_1\times A_2 \arrow[rr, bend left=30, "\psi"]\rar{1\times m} & A_{1}\times A_{2} \rar{\phi} & A
\end{tikzcd}
\]
given by the composition of two isogenies. Let $\proj: A_1 \times A_2 \rightarrow A_2$ be the projection map. We wish to show that $\proj(\psi^{-1}(C))$ is equal to $\proj(\psi^{-1}(O))$ where $O$ is the identity element of $A$. The key point is that since $C\subset A_{1}\subset A$ if $\phi(a_{1},a_2)\in C$ then $a_1+a_2\in A_1$ implying that $a_2=(a_1+a_2)-a_1$ is contained in $A_{1}$. Phrased differently $\phi^{-1}(C)$ is equal to $\{(a_1-a_2,a_2) \; | \; a_{1}\in C, \;\;a_{2}\in A_{1}\cap A_{2}\}$. The equality now follows from the fact that the kernel  of $\phi$ is $\{(a,-a) \; | \; a\in A_1\cap A_2\}$.
 
Since $\phi$ is an isogeny the kernel  of $\phi$, which is $\{(a,-a) \; | \; a\in A_1\cap A_2\}$, is finite. Thus, $A_{1}\cap A_{2}$ is a finite set, and moreover, $A_{1}\cap A_{2}$ is non-empty since the identity element of $A$ is contained in $A_1 \cap A_2$. As $m$ is relatively prime to the characteristic of our ground field, multiplication by $m$ is a finite map of degree $m^{2n}$ where $n$ is the dimension of $A_2$. Thus, $\proj(\psi^{-1}(C))$ is a finite set of size $m^{2n}|A_1 \cap A_2|>1$, which implies that $\psi^{-1}(C)$ is not geometrically connected. However, applying Lemma \ref{lemma:hyersurface-finite-maps} repeatedly shows that $\psi^{-1}(C)$ is geometrically connected, providing a contradiction. So, we conclude $A_1=A$ and $\pi$ is surjective.
\end{proof}


\section{Effective Bertini Theorem over Finite Fields}\label{sec:prob-curves}

In this section we establish an effective Bertini Theorem over finite fields, proving Theorem~B. This section also contains a technical version of Theorem~B, Proposition~\ref{prop:effective-bucur-kedlaya}, where all constants are explicitly stated.

A key ingredient in the proof of these results is recent work of Bucur and Kedlaya \cite{bucurKedlaya12} which characterizes the probability that the intersection of an $n$-dimensional quasi-projective subscheme $X$ with $k$ randomly chosen hypersurfaces of given degrees is smooth of dimension $n-k$. While Bucur and Kedlaya's result is not itself effective, it does contain an explicit error term. We carefully analyze this error term to produce an effective Bertini Theorem.

Before stating their result and using it to prove Theorem~B, we fix a bit of notation. Let $S=\fF_{q}[x_0,\ldots,x_r]$ be the homogeneous coordinate ring of $\P^{r}_{\fF_q}$. Given a tuple $\dd=(d_1,\ldots,d_k)\in \N^k$ we set
\[
S_{\dd}=S_{d_1}\oplus S_{d_2}\oplus \cdots \oplus S_{d_k},
\]
where $S_{d_i}$ is the $\fF_{q}$-vector space of homogeneous polynomials of degree $d_i$ in $S$. Further, given an element $\ff=(f_1,\ldots,f_k)\in S_{\dd}$ we write $\V(\ff)$ for $\V(f_1,\ldots,f_k)\subset \P^r_{\fF_{q}}$. The probability that $k$ uniformly chosen vectors in $\fF_{q}^n$ are linearly independent is denoted as follows
\[
L(q,n,k)=\prod_{j=0}^{k-1}\left(1-q^{-(n-j)}\right).
\]
With this notation in hand we now state Bucur and Kedlaya's result. 

\begin{theorem}\label{thm:bucur-kedlaya}\cite[Theorem~1.2]{bucurKedlaya12}
Let $X\subset \P^{r}_{\fF_q}$ be a smooth quasi-projective subscheme of dimension $n\geq0$ over a finite field $\fF_q$ of characteristic $p$. Choose an integer $k\in\{1,\ldots,n-1\}$, a degree sequence $\dd=(d_1\leq d_2\leq \cdots \leq d_k)$, and  set
\[
\cP_{\dd} = \left\{ \ff\in S_{\dd} \quad \bigg| \quad
\begin{matrix}
\text{$X\cap \V(\ff)$ has dimension $n-k$}\\
\text{and is smooth}
\end{matrix}
\right\}.
\]
Then
\begin{align}\label{Probinequality}
\left|\frac{\#\cP_{\dd}}{\#S_{\dd}}-\prod_{x\in X}\left(1-q^{-k\deg(x)}+q^{-k\deg(x)}L\left(q^{\deg(x)},n,k\right)\right)\right|
\leq &\, 2^{n+2}\deg(X)kq^{-\delta} \nonumber \\ &+(r+1)kr^{n}\deg(X)(n+1)d_{k}^{n}q^{\frac{-d_1}{\max\{n+1,p\}}},
\end{align} 
where
\[
\delta=(2k-1)\left(1+\left\lfloor\frac{1}{n}\log_{q}\frac{d_1+1}{(n+1)2^{n+1}}\right\rfloor\right).
\]
\end{theorem}

To prove Theorem~B, we need to control the Euler product appearing in the above theorem. In general this is difficult. For example, \cite[pg. 544]{bucurKedlaya12} presents numerical evidence suggesting it cannot be interpreted as a zeta function. But we are able to provide a lower bound for it in terms of a zeta function value. 

\begin{prop}\label{prop:euler-product-bound}
Let $X\subset \P^{r}_{\fF_{q}}$ be a smooth quasi-projective subscheme of dimension $n\geq0$ defined over a finite field $\fF_{q}$. Fix $1\leq k \leq n-1$. If $q\geq 3$ then
\[
\zeta_{X}\left(n+\tfrac{1}{2}\right)^{-1}\leq\prod_{x\in X}\left(1-q^{-k\deg(x)}+q^{-k\deg(x)}L\left(q^{\deg(x)},n,k\right)\right),
\]
and if $q=2$ then
\[
2^{-\#X\left(\fF_{2}\right)}\zeta_{X}\left(n+\tfrac{1}{2}\right)^{-1}\leq\prod_{x\in X}\left(1-q^{-k\deg(x)}+q^{-k\deg(x)}L\left(q^{\deg(x)},n,k\right)\right).
\]
\end{prop}

To prove this proposition, we need two lemmas.

\begin{lemma}\label{lem:bound-on-product}
If $\{a_i\}_{i=1}^{t}$ is a sequence of real numbers such that $0< a_i< 1$ then
\[
1-\sum_{i=1}^ta_i\leq \prod_{i=1}^t\left(1-a_i\right) < 1.
\]
\end{lemma}

\begin{proof}
The upper bound is immediate from the fact that $0 < 1-a_i < 1$ for all $i$. For the lower bound we proceed by induction on $t$ with the case when $t=1$ being clear. In the general case by induction we assume
\[
1-\sum_{i=1}^{t-1}a_i\leq \prod_{i=1}^{t-1}\left(1-a_i\right),
\]
and multiplying both sides by $(1-a_t)$ gives
\begin{align*}
1-\sum_{i=1}^{t}a_i\leq 1-\sum_{i=1}^{t}a_i+a_t\left(\sum_{i=1}^{t-1}a_i\right)=\left(1-\sum_{i=1}^{t-1}a_i\right)\left(1-a_t\right) \leq \prod_{i=1}^{t}\left(1-a_i\right), 
\end{align*}
which completes the inductive step.
\end{proof}

\begin{lemma}\label{lem:bound-on-L-term}
Fix $1\leq k \leq n-1$. If either $q \geq 3$ and $t \geq 1$ or if $q=2$ and $t>1$ then
\[
1-q^{-\left(n-k+\tfrac{1}{2}\right)t}\leq L\left(q^{t},n,k\right).
\]
Moreover, if $q=2$ and $t=1$ we have
\[
\tfrac{1}{2}\left(1-2^{-\left(n-k+\tfrac{1}{2}\right)}\right)\leq L\left(2,n,k\right).
\]
\end{lemma}

\begin{proof}
Combining the definition of $L\left(q^{t},n,k\right)$ with Lemma~\ref{lem:bound-on-product} we know that
\[
1-\sum_{i=0}^{k-1}q^{-(n-i)t} \leq \prod_{i=0}^{k-1}\left(1-q^{-(n-i)t}\right) = L\left(q^{t},n,k\right).
\]
Since the left-hand side is a geometric sum we may rewrite this inequality as
\begin{align*}
1-\frac{q^{-(n-k+1)t}-q^{-(n+1)t}}{1-q^{-t}}
= 1-\sum_{i=0}^{k-1}q^{-(n-i)t} \leq L\left(q^{t},n,k\right),
\end{align*}
which we may further simplify to
\[
1-\frac{q^{-(n-k+1)t}}{1-q^{-t}}\leq 1-\frac{q^{-(n-k+1)t}-q^{-(n+1)t}}{1-q^{-t}} \leq L\left(q^{t},n,k\right).
\]

Now we shift to showing that in the cases when $q\geq 3$ and $t \geq 1$ or $q=2$ and $t>1$
\[
1 - q^{-\left(n-k+\tfrac{1}{2}\right)t}\leq 1 - \frac{q^{-(n-k+1)t}}{1-q^{-t}}.
\]
Rearranging the terms, one sees the above inequality is equivalent to 
\begin{equation}\label{eq:1}
\frac{q^{-\frac{t}{2}}}{1-q^{-t}}=\frac{q^{-\left(n-k+1\right)}q^{\left(n-k+\tfrac{1}{2}\right)t}}{1-q^{-t}}\leq 1.
\end{equation}
Notice the above inequality is equivalent to $q^{t}-q^{\tfrac{t}{2}}-1\geq0$. Since $x^2-x-1\geq0$ for all $x\geq \tfrac{1}{2}(1+\sqrt{5})$ it is thus enough to have $q^{\tfrac{t}{2}}\geq \tfrac{1}{2}(1+\sqrt{5})$; however, this is true since $q\geq3$ and $t\geq1$ and so $q^{\tfrac{t}{2}}\geq \sqrt{3} >\tfrac{1}{2}(1+\sqrt{5})$.

Finally we focus on the remaining case, when $q=2$ and $t=1$. From our work above we know
\[
1-\frac{q^{-(n-k+1)t}}{1-q^{-t}}\leq 1-\frac{q^{-(n-k+1)t}-q^{-(n+1)t}}{1-q^{-t}} \leq L\left(q^{t},n,k\right),
\]
and so it is enough to show that
\[
\tfrac{1}{2}\left(1-2^{-\left(n-k+\tfrac{1}{2}\right)}\right)\leq1-2^{-(n-k)}=1-\frac{2^{-(n-k+1)}}{1-2^{-1}}.
\]
Rearranging the terms, this inequality is equivalent to 
\[
1\leq2^{n-1-k}+2^{-\frac{3}{2}}.
\]
The right-hand side is minimized when $k=n-1$, in which case it is equal to $1+2^{-\frac{3}{2}}$, and so the desired inequality holds for all $1\leq k\leq n-1$.
\end{proof}

\begin{proof}[Proof of Proposition~\ref{prop:euler-product-bound}]
By Lemma~\ref{lem:bound-on-L-term}, if $q\geq 3$ then
\begin{align*}
\zeta_{X}\left(n+\tfrac{1}{2}\right)^{-1}=\prod_{x\in X}\left(1-q^{-\left(n+\tfrac{1}{2}\right)\deg(x)}\right)&=\prod_{x\in X}\left(1-q^{-k\deg(x)}+q^{-k\deg(x)}\left(1-q^{-\left(n-k+\tfrac{1}{2}\right)\deg(x)}\right)\right)\\
&\leq\prod_{x\in X}\left(1-q^{-k\deg(x)}+q^{-k\deg(x)}L\left(q^{\deg(x)},n,k\right)\right).
\end{align*}
Similarly in the $q=2$ case for points $x\in X$ of degree not one Lemma~\ref{lem:bound-on-L-term} tells us that
\begin{align}\label{ref:ineq-not-degree-one}
\prod_{\substack{x\in X\\ \deg(x)\neq1}}\left(1-q^{-\left(n+\tfrac{1}{2}\right)\deg(x)}\right)
\leq\prod_{\substack{x\in X\\ \deg(x)\neq1}}\left(1-q^{-k\deg(x)}+q^{-k\deg(x)}L\left(q^{\deg(x)},n,k\right)\right).
\end{align}
On the other hand for points $x\in X$ of degree one Lemma~\ref{lem:bound-on-L-term} implies that
\begin{align}\label{ref:ineq-degree-one}
\prod_{\substack{x\in X\\ \deg(x)=1}}\tfrac{1}{2}\left(1-q^{-\left(n+\tfrac{1}{2}\right)}\right)&=\prod_{\substack{x\in X\\ \deg(x)=1}}\left(\tfrac{1}{2}-\tfrac{1}{2}q^{-k\deg(x)}+\tfrac{1}{2}q^{-k\deg(x)}\left(1-q^{-\left(n-k+\tfrac{1}{2}\right)\deg(x)}\right)\right)\\ 
&\leq \prod_{\substack{x\in X\\ \deg(x)=1}}\left(1-q^{-k\deg(x)}+\tfrac{1}{2}q^{-k\deg(x)}\left(1-q^{-\left(n-k+\tfrac{1}{2}\right)\deg(x)}\right)\right)\\&\leq\prod_{\substack{x\in X\\ \deg(x)=1}}\left(1-q^{-k\deg(x)}+q^{-k\deg(x)}L\left(q^{\deg(x)},n,k\right)\right).
\end{align}
Multiplying Inequality~\eqref{ref:ineq-not-degree-one} and Inequality~\eqref{ref:ineq-degree-one} gives the result in the case when $q=2$.
\end{proof}

We now prove the following proposition, which is a more precise version of Theorem~B.

\begin{prop}\label{prop:effective-bucur-kedlaya}
Let $X\subset \P^{r}_{\fF_{q}}$ be a smooth quasi-projective subscheme of dimension $n\geq2$ defined over a finite field $\fF_{q}$ of characteristic $p$. Fix $1\leq k \leq n-1$. Under either of the following circumstances
\begin{enumerate}
\item $q\geq3$, and $d \in \N$ satisfies the inequality
\begin{equation}\label{degreeinequality}
k2^{n+2k+1+\frac{2k-1}{n}}(n+1)^{1+\frac{2k-1}{n}}(r+1)r^{n}\deg X\zeta_{X}\left(n+\tfrac{1}{2}\right)< \frac{q^{\frac{d}{\max\{n+1,p\}}}\left(d^{\frac{2k-1}{n}}+1\right)}{d^{n}+d^{n+\frac{2k-1}{n}}+q^{\frac{d}{\max\{n+1,p\}}}},
\end{equation}
\item or  $q=2$, and $d\in \N$ satisfies the inequality
\begin{equation*}
k2^{n+2k+1+\frac{2k-1}{n}+\#X\left(\fF_{2}\right)}(n+1)^{1+\frac{2k-1}{n}}(r+1)r^{n}\deg X\zeta_{X}\left(n+\tfrac{1}{2}\right)<  \frac{q^{\frac{d}{\max\{n+1,p\}}}\left(d^{\frac{2k-1}{n}}+1\right)}{d^{n}+d^{n+\frac{2k-1}{n}}+q^{\frac{d}{\max\{n+1,p\}}}},
\end{equation*}
\end{enumerate}
there exist homogeneous polynomials $f_1,\ldots,f_{k}\in \fF_{q}[x_0,\ldots,x_r]$ of degree $d$ such that $X\cap \V\left(f_1,\ldots,f_{k}\right)$ is smooth of dimension $n-k$. Moreover, if $X$ is projective and geometrically connected then $X\cap \V(f_1, \ldots, f_k)$ is also geometrically connected.
\end{prop}

\begin{proof}
Setting $d=d_1=d_2\cdots=d_{k}$ we wish to show that $\frac{\#\cP_{\dd}}{\#S_{\dd}}>0$, which since $\frac{\#\cP_{\dd}}{\#S_{\dd}}\geq0$ is equivalent to showing that $\frac{\#\cP_{\dd}}{\#S_{\dd}}\neq0$. By Theorem~\ref{thm:bucur-kedlaya}:
\[
\left|\frac{\#\cP_{\dd}}{\#S_{\dd}}-\prod_{x\in X}\left(1-q^{-k\deg(x)}+q^{-k\deg(x)}L\left(q^{\deg(x)},n,k\right)\right)\right|
\leq  2^{n+2}\deg(X)kq^{-\delta}+(r+1)kr^{n}\deg(X)(n+1)d_{k}^{n}q^{\frac{-d_1}{\max\{n+1,p\}}},
\]
and so to show that $\frac{\#\cP_{\dd}}{\#S_{\dd}}\neq0$ it is enough to show that
\begin{equation}\label{main-eqn:buccur-kedlaya}
2^{n+2}\deg(X)kq^{-\delta}+(r+1)kr^{n}\deg(X)(n+1)d^{n}q^{\frac{-d}{\max\{n+1,p\}}} < \prod_{x\in X}\left(1-q^{-k\deg(x)}+q^{-k\deg(x)}L\left(q^{\deg(x)},n,k\right)\right).
\end{equation}
Using Proposition~\ref{prop:euler-product-bound} to bound the right-hand side of the above inequality it is enough to show that:
\begin{equation}\label{eqn:left-hand-side-I}
\begin{split}
2^{n+2}\deg(X)kq^{-\delta}+(r+1)kr^{n}\deg(X)(n+1)d^{n}q^{\frac{-d}{\max\{n+1,p\}}} <  \zeta_{X}\left(n+\tfrac{1}{2}\right)^{-1} \quad \quad\quad (q\neq2) \\
2^{n+2}\deg(X)kq^{-\delta}+(r+1)kr^{n}\deg(X)(n+1)d^{n}q^{\frac{-d}{\max\{n+1,p\}}} < 2^{-\#X\left(\fF_2\right)}\zeta_{X}\left(n+\tfrac{1}{2}\right)^{-1}  \quad \quad\quad (q=2).\\
\end{split}
\end{equation}

We now proceed by bounding the left-hand side of Inequality~\eqref{eqn:left-hand-side-I}. Since $r$, $k$, and $n$ are positive constants and $r\geq1$, the left-hand side of Inequality~\eqref{eqn:left-hand-side-I} satisfies the following:
\begin{equation}\label{eqn:right-hand-side-I}
2^{n+2}\deg(X)kq^{-\delta}+(r+1)kr^{n}\deg(X)(n+1)d^{n}q^{\frac{-d}{\max\{n+1,p\}}} \leq k2^{n+2}(n+1)(r+1)r^{n}\deg X \left[q^{-\delta}+d^{n}q^{\frac{-d}{\max\{n+1,p\}}}\right].
\end{equation}
With $\delta$ as in Theorem~\ref{thm:bucur-kedlaya} we may bound $\delta$ as follows:
\[
\frac{2k-1}{n}\log_{q}\frac{d+1}{(n+1)2^{n+1}} = (2k-1)\left(1+\frac{1}{n}\log_{q}\frac{d+1}{(n+1)2^{n+1}}-1\right) \leq (2k-1)\left(1+\left\lfloor\frac{1}{n}\log_{q}\frac{d+1}{(n+1)2^{n+1}}\right\rfloor\right)= \delta.
\]
This allows us to bound $q^{-\delta}$ from above, giving an upper bound for the right-hand side of Inequality~\eqref{eqn:right-hand-side-I}:
\begin{equation}\label{eqn:right-hand-side-II}
k2^{n+2}(n+1)(r+1)r^{n}\deg X\left[q^{-\delta}+d^{n}q^{\frac{-d}{\max\{n+1,p\}}}\right]\leq k2^{n+2}(n+1)(r+1)r^{n}\deg X\left[\left(\tfrac{(n+1)2^{n+1}}{d+1}\right)^{\frac{2k-1}{n}}+d^{n}q^{\frac{-d}{\max\{n+1,p\}}}\right].
\end{equation}
Since $n$ is a positive constant, we may give an upper bound to the right-hand side of Inequality~\eqref{eqn:right-hand-side-II} by ``pulling out'' $((n+1)2^{n+1})^{\frac{2k-1}{n}}$. Further since $d\geq1$ we may bound $(d+1)^{\frac{2k-1}{n}}$ below by $d^{\frac{2k-1}{n}}+1$. This allows us to bound the right-hand side of Inequality~\eqref{eqn:right-hand-side-II} from above by the following:
\begin{equation}\label{eqn:right-hand-side-III}
k2^{n+2k+1+\frac{2k-1}{n}}(n+1)^{1+\frac{2k-1}{n}}(r+1)r^{n}\deg X\left[\frac{d^{n}+d^{n+\frac{2k-1}{n}}+q^{\frac{d}{\max\{n+1,p\}}}}{q^{\frac{d}{\max\{n+1,p\}}}\left(d^{\frac{2k-1}{n}}+1\right)}\right].
\end{equation}
Combining Inequalities~\eqref{eqn:right-hand-side-I}, \eqref{eqn:right-hand-side-II}, and \eqref{eqn:right-hand-side-III} we get our final upper bound for the left-hand side of Inequality~\eqref{eqn:left-hand-side-I}:
\begin{equation}\label{eq:9}
2^{n+2}\deg Xkq^{-\delta}+(r+1)kr^{n}\deg X (n+1)d^{n}q^{\frac{-d}{\max\{n+1,p\}}} \leq k2^{n+2k+1+\frac{2k-1}{n}}(n+1)^{1+\frac{2k-1}{n}}(r+1)r^{n}\deg X\left[\frac{d^{n}+d^{n+\frac{2k-1}{n}}+q^{\frac{d}{\max\{n+1,p\}}}}{q^{\frac{d}{\max\{n+1,p\}}}\left(d^{\frac{2k-1}{n}}+1\right)}\right].
\end{equation}

So by Inequalities~\eqref{eqn:left-hand-side-I} and \eqref{eq:9} if $d\in \N$ satisfies:
\begin{equation*}\label{eq:10}
\begin{split}
k2^{n+2k+1+\frac{2k-1}{n}}(n+1)^{1+\frac{2k-1}{n}}(r+1)r^{n}\deg X\left[\frac{d^{n}+d^{n+\frac{2k-1}{n}}+q^{\frac{d}{\max\{n+1,p\}}}}{q^{\frac{d}{\max\{n+1,p\}}}\left(d^{\frac{2k-1}{n}}+1\right)}\right] <  \zeta_{X}\left(n+\tfrac{1}{2}\right)^{-1} \quad \quad\quad (q\neq2) \\
k2^{n+2k+1+\frac{2k-1}{n}}(n+1)^{1+\frac{2k-1}{n}}(r+1)r^{n}\deg X\left[\frac{d^{n}+d^{n+\frac{2k-1}{n}}+q^{\frac{d}{\max\{n+1,p\}}}}{q^{\frac{d}{\max\{n+1,p\}}}\left(d^{\frac{2k-1}{n}}+1\right)}\right]< 2^{-\#X\left(\fF_2\right)}\zeta_{X}\left(n+\tfrac{1}{2}\right)^{-1}  \quad \quad\quad (q=2).\\
\end{split}
\end{equation*}
then such $d$ also satisfies Inequality~\eqref{main-eqn:buccur-kedlaya} meaning that $\frac{\#\cP_{\dd}}{\#S_{\dd}}>0$.

Finally, since $X$ is smooth it is geometrically reduced \cite[\href{https://stacks.math.columbia.edu/tag/056T}{Lemma 056T}]{stacks-project}. In particular, if $X$ is geometrically connected then it is geometrically integral. Thus, if $X$ is also projective, then since $n\geq2$ and $n-k\geq1$ we may inductively apply \cite[Section III, Corollary 7.9]{hartshorne} to deduce that $X\cap \V\left(f_1,\ldots,f_{k}\right)$ is geometrically connected.
\end{proof}

\begin{remark}
The inequalities appearing in Proposition~\ref{prop:effective-bucur-kedlaya} are eventually true for $d$ sufficiently large since the right-hand sides tend to infinity as $d\to\infty$ while the left-hand side is independent of $d$.
\end{remark}

\begin{proof}[Proof of Theorem~B]
Since $\#X\left(\fF_{2}\right)\leq \#\P^{r}\left(\fF_{2}\right) = 2^{r+1}-1$, we can bound $\#X\left(\fF_{2}\right)$ in terms of just $r$. Thus, by Proposition~\ref{prop:effective-bucur-kedlaya} if we let
\[
C_{r,q}=
\begin{cases}
2^{3r+1}(r+1)^{5}r^{r} &\mbox{if } q\neq2 \\
2^{3r+2^{r+1}}(r+1)^{5}r^{r} &\mbox{if } q=2
\end{cases}
\]
there exist homogeneous polynomials $f_1,\ldots,f_{k}\in \fF_{q}[x_0,\ldots,x_r]$ of degree $d$ such that $X\cap \V\left(f_1,\ldots,f_{k}\right)$ is smooth of dimension $n-k$, which is geometrically connected if $X$ is projective and geometrically connected.
\end{proof}


\begin{remark}\label{rem:non-explicit-bound}
Regarding Theorem~B, Poonen has pointed out to us, in personal communication, that by using a noetherian induction argument, one can show the existence of a bound dependent solely on $r$ and the degree of $X$. While such a bound would be ineffective, it would be independent of $q$ and $n$.


\end{remark}


\section{Smooth Curves of Bounded Genus and Degree}\label{sec:bound-genus}

We now bound the degree and genus of the smooth curves $C \subset X$ we constructed in the previous section.

\begin{prop}\label{cor:degree-bound}
Let $X\subset \P^{r}_{\fF_{q}}$ be a smooth projective subscheme of dimension $n\geq2$ defined over a finite field $\fF_{q}$ of characteristic $p$. Under either of the following circumstances
\begin{enumerate}
\item $q\geq3$, and $d \in \N$ satisfies the inequality
\[
2^{3n+3}\deg(X)n^4r^{n+1}\zeta_{X}\left(n+\tfrac{1}{2}\right)\leq \frac{q^{\frac{d}{\max\{n+1,p\}}}\left(d^{\tfrac{1}{2}}+1\right)}{d^{n+2}+d^n+q^{\frac{d}{\max\{n+1,p\}}}},
\]
\item or $q=2$, and $d\in \N$ satisfies the inequality
\[
2^{3n+\#X\left(\fF_2\right)+3}\deg(X)n^4r^{n+1}\zeta_{X}\left(n+\tfrac{1}{2}\right)\leq \frac{q^{\frac{d}{\max\{n+1,p\}}}\left(d^{\tfrac{1}{2}}+1\right)}{d^{n+2}+d^n+q^{\frac{d}{\max\{n+1,p\}}}},
\]
\end{enumerate}
there exist homogeneous polynomials $f_1,\ldots,f_{n-1}\in \fF_{q}[x_0,\ldots,x_r]$ of degree $d$ such that $X\cap \V\left(f_1,\ldots,f_{n-1}\right)$ is a smooth curve and $\deg(C)=\deg(X)d^{n-1}$. Moreover, if $X$ is projective and geometrically connected then $X\cap \V(f_1, \ldots, f_{n-1})$ is also geometrically connected.
\end{prop}

\begin{proof}

As $n,r \ge 1$ note that $(n-1)(n+1)^{3-\frac{3}{n}}(r+1)r^n \le 4n^4r^{n+1}$, and so 
\begin{align*}
2^{3n-\frac{3}{n}+1}\deg(X)(n-1)(n+1)^{3-\frac{3}{n}}(r+1)r^n\zeta_{X}\left(n+\frac{1}{2}\right) &\leq 2^{3n+3}\deg(X)n^4r^{n+1}\zeta_{X}\left(n+\tfrac{1}{2}\right)\\
2^{3n-\frac{3}{n}+\#X\left(\fF_{2}\right)+1}\deg(X)(n-1)(n+1)^{3-\frac{3}{n}}(r+1)r^{n}\zeta_{X}\left(n+\tfrac{1}{2}\right) &\leq 2^{3n+\#X\left(\fF_2\right)+3}\deg(X)n^4r^{n+1}\zeta_{X}\left(n+\tfrac{1}{2}\right).
\end{align*}
Moreover, we see that
\[
 \frac{q^{\frac{d}{\max\{n+1,p\}}}\left(d^{\frac{1}{2}}+1\right)}{d^{n+2}+d^n+q^{\frac{d}{\max\{n+1,p\}}}}
\leq \frac{q^{\frac{d}{\max\{n+1,p\}}}\left(d^{2-\frac{3}{n}}+1\right)}{d^{n}+d^{n+2-\frac{3}{n}}+q^{\frac{d}{\max\{n+1,p\}}}}.
\]
Thus, given $d\in \N$ as in the statement of this proposition then applying Proposition~\ref{prop:effective-bucur-kedlaya} in the case when $k=n-1$ there exist the desired homogeneous polynomials $f_1,\ldots,f_{n-1}\in \fF_{q}[x_0,\ldots,x_r]$ of degree $d$ such that $X\cap \V\left(f_1,\ldots,f_{n-1}\right)$ is a smooth curve. Further, Bezout's Theorem~\cite[Proposition~8.4]{fulton} implies
\[
\deg(C)= \deg(X)\prod_{i=1}^{n-1}\deg(f_i)=\deg(X)d^{n-1}.
\]
Finally, as stated in Proposition~\ref{prop:effective-bucur-kedlaya} if $X$ is projective and geometrically connected then $X\cap \V\left(f_1,\ldots,f_{n-1}\right)$ is geometrically connected.

\end{proof}

To show the existence of smooth connected curves with bounded genus, we use a classical theorem of Castelnuovo which gives an upper bound on the genus of an irreducible, smooth, non-degenerate curve $X\subset \P^r$ in terms of $\deg X$ and $r$. (Recall a scheme $X\subset \P^r$ is non-degenerate if it is not contained in any hyperplane.) 

\begin{prop}\label{prop:bounded-genus}
Let $X\subset \P^{r}_{\fF_{q}}$ be a smooth non-degenerate projective geometrically connected subscheme of dimension $n\geq2$ defined over a finite field $\fF_{q}$ of characteristic $p$. 
If $d\geq2$ is a natural number satisfying the condition in Proposition \ref{cor:degree-bound}, then there exists a smooth geometrically connected non-degenerate curve $C\subset X$ such that
\[
g(C)\leq 
\left\lfloor \frac{\deg(X)d^{n-1}-1}{r-1}\right\rfloor \left(\deg(X)d^{n-1}-\frac{\left\lfloor\frac{\deg(X)d^{n-1}-1}{r-1}\right\rfloor+1}{2}(r-1)-1\right).
\]
\end{prop}

\begin{proof}
By Proposition~\ref{cor:degree-bound}, for such $d\geq2$ there exists a smooth geometrically connected curve $C\subset X$ with $\deg(C)= \deg(X)d^{n-1}$. To show that $C$ is non-degenerate it is enough, by induction, to show that $X\cap\V(f_1)$ is non-degenerate. If $X\cap\V(f_1)$ were degenerate, and so contained in a linear subspace $L\subset \P^{r}_{\fF_{q}}$, then $X\cap\V(f_1)\subset X\cap L$, and since $X$ itself is non-degenerate both $X\cap\V(f_1)$ and $X\cap L$ have dimension $n-1$. However, by Bezout's Theorem~\cite[Proposition~8.4]{fulton} the degree of $X\cap\V(f_1)$ is equal to $\deg(X)d$, which since $d\geq2$ is strictly larger than $\deg(X\cap L)=\deg(X)$, giving a contradiction. Finally, applying Castelnuovo's genus bound \cite[pg. 40]{harris81} to $C$ gives the stated result.
\end{proof}

We conclude this section with the proof of the statement in Theorem~A for general abelian varieties.

\begin{proof}[Proof of Theorem~A (General Case)]
Since $n\leq r$ by Propositions~\ref{prop:dominate-map-av} and \ref{prop:bounded-genus}, it is enough to show that if $q=2$ then $\#A\left(\fF_{2}\right)$ is bounded by a constant depending only on $n$ and $r$. This follows immediately from the Weil bounds~\cite[pg. 3]{WBound}, which states that $\#A\left(\fF_{2}\right)$ is bounded above by $(3+2\sqrt{2})^n$. Thus, the result follows with $C_{r,q}$ defined as:
\[
C_{r,q}=\begin{cases}
2^{3r+3}r^{r+5} &\mbox{if } q\neq 2\\
2^{3r+3+(3+2\sqrt{2})^r}r^{r+5} &\mbox{if } q= 2
\end{cases}.
\]
\end{proof}

\begin{remark}
Notice the dependence of $C_{r,q}$ on $q$ is really only dependence on whether or not $q=2$. Thus, one can easily make $C_{r,q}$ independent of $q$ by adding in the appropriate factors of $2$.  
\end{remark}


\section{The Case when A is Simple}\label{sec:simple-case}

When $A$ is a simple abelian variety, our general bound can be simplified as was stated in the second part of Theorem~A. This is possible because when $A$ is simple, almost any curve on $A$, even if it is reducible, non-reduced, or non-smooth, gives rise to a covering of $A$ by a Jacobian. 

In particular, suppose that $C\subset A$ is any curve on $A$. By taking an irreducible component of $C$ considered with the reduced subscheme structure without loss of generality we may assume that $C$ is irreducible and reduced. Now taking the normalization of this irreducible reduced curve $C$ results in a smooth irreducible curve $\tilde{C}$, which maps non-trivially to $A$. The universal property of Jacobian varieties in turn gives a nonconstant map $\Jac(\tilde{C}) \rightarrow{}A$, and as $A$ is simple this map must be surjective.

Thus, in the simple case, constructing curves whose Jacobians dominate $A$ is easier. One only needs the existence of a (possibly non-smooth, non-reduced, or reducible) curve $C$ contained in $A$. So it is sufficient to find homogeneous polynomials $f_1, \ldots, f_{n-1}$, which cut out any curve on $A$. This allows us to choose the $f_1,\ldots,f_{n-1}$ to be of smaller degree, improving bound. 

\begin{prop}\label{prop:degree-curve-simple}
Let $X\subset \P^{r}_{\fF_{q}}$ be a smooth projective subscheme of dimension $n$ defined over a finite field $\fF_{q}$. If $d\in \N$ satisfies the following inequality
\[
\deg(X)\leq \frac{(d-1)q^{\tfrac{1}{2}(d+1)(d+2)}}{d^{n-1}-1},
\]
then there exist homogeneous polynomials $f_1,\ldots,f_{n-1}\in \fF_{q}[x_0,...,x_r]$ of degree $d$ such that $C=X\cap \V\left(f_1,...,f_{n-1}\right)$ is a curve and $\deg(C)=\deg(X)d^{n-1}.$
\end{prop}

\begin{proof}
By combining the given inequality on $d$ with Proposition 5.1 of \cite{bruce16} in the case when $k=n-2$
we can find homogeneous polynomials $f_1,..., f_{n-1}$ of degree $d$ where $X \cap V(f_1,...,f_{n-1})$ has dimension $1$. Bezout's Theorem~\cite[Proposition~8.4]{fulton} then gives $\deg(C)= \deg(X)\prod_{i=1}^{n-1}\deg(f_i)=\deg(X)d^{n-1}.$
\end{proof}

To finish the proof of Theorem~A, we must be able to bound the genus of the normalization $\tilde{C}$ in terms of the degree of $C$. As the genus of $\tilde{C}$ is bounded above by the arithmetic genus of $C$ \cite[Exercise IV.1.8]{hartshorne} it is enough to bound the arithmetic genus of $C$. (We write $p_{a}(C)$ for the arithmetic genus of a curve $C$.)

As before, the idea is to use a degree-genus bound. However, since the curves arising in Proposition~\ref{prop:degree-curve-simple} need not be smooth we cannot use Castelnuovo's genus bound. Instead we prove a less sharp, but more general bound by combining a lower bound on the Hilbert function/polynomial with a bound on the Castelnuovo-Mumford regularity.

\begin{lemma}\label{lem:hilbert-function-bound}
If $C\subset \P^r_{\kk}$ is a curve with homogeneous coordinate ring $R$, then $\dim R_{d}\geq d+1$ for any $d\in \N$.
\end{lemma}

\begin{proof}
Since base change does not affect the Hilbert function, without loss of generality, we may suppose that $\kk$ is algebraically closed. Since $\kk$ is infinite, there exists a linear form $\ell\in R$, which gives rise to the short exact sequence
\[
\begin{tikzcd}[column sep = 3em]
0\rar{}&R(-1)\rar{\cdot \ell}& R\rar{}& R/\langle \ell \rangle\rar{}&0.
\end{tikzcd}
\]
Using the additivity of the Hilbert function, we see that $\dim R_{d}=\sum_{k=0}^{d}\dim \left(R/\langle \ell \rangle\right)_{k}$ for any $d\in \N$, and since $R/\langle \ell \rangle$ is one-dimensional, the result now follows by noting that $\dim (R/\langle \ell \rangle)_k\geq1$ for all $k\geq0$. 
\end{proof}

With this lemma in hand, we prove a more general genus-degree bound that applies to all geometrically connected reduced equidimensional curves.

\begin{lemma}\label{lem:genus-reg-bound}
If $C\subset \P^r_{\kk}$ is a geometrically connected reduced curve, then 
\[
p_a(C) \leq \deg(C)(\deg(C)+1)-2.
\]
\end{lemma}

\begin{proof}
Since the hypotheses are stable under base change, without loss of generality, we may suppose that $\kk$ is algebraically closed and that $C$ is connected. Let $R$ be the homogeneous coordinate ring of the curve $C$. The Hilbert polynomial $P_C(t)$ of the curve $C$ is equal to $\deg(C)t + 1 -p_a(C)$. For any $t \geq \reg(C)$, the Hilbert function and Hilbert polynomial agree \cite[Theorem 4.2]{eisenbud05}. Thus, if $t\geq \reg(C)$ then by Lemma~\ref{lem:hilbert-function-bound}:
\[
t+1\leq \dim R_{t} =  P_{C}(t) = \deg(C)t+1-p_a(C).
\]
Results of Giaimo imply that $\reg(C) \leq \deg(C)+2$ \cite{giaimo06}. Plugging $t=\deg(C)+2$ into the above inequality yields:
\[
\deg(C)+3\leq \dim R_{\deg(C)+2} = \deg(C)\left(\deg(C)+2\right)+1-p_a(C).
\]
The result now follows from rearranging the above inequality.
\end{proof}

\begin{remark}
Not only does the bound from Lemma~\ref{lem:genus-reg-bound} apply to non-smooth curves, it also applies to degenerate curves, i.e. curves lying in a hyperplane in $\P^r_{\kk}$. In fact, such curves attain the maximal values, as any degree $d$ planar curve will have the maximal possible arithmetic genus.
\end{remark}

Finally, we conclude the proof of Theorem~A.

\begin{proof}[Proof of Theorem~A (Simple Case)]
By Proposition~\ref{prop:degree-curve-simple}, there exist homogeneous polynomials $f_1,\ldots,f_{n-1}\in \fF_{q}[x_0,x_1,\ldots,x_r]$ of degree $d$ such that $C=A\cap \V\left(f_1,f_2,\ldots,f_{n-1}\right)$ is a curve with $\deg(C)=\deg(A)d^{n-1}$. Let $C'_{\red}\subset C$ be an irreducible component of $C$ considered with the reduced subscheme structure. As noted in the beginning of this section, if $\tilde{C}'_{\red}$ is the normalization of $C'_{\red}$, then since $A$ is simple the map $\Jac(\tilde{C}'_{\red})\rightarrow A$ coming from the universal property of Jacobians is surjective. Hence it is enough to bound the genus of $\tilde{C}'_{\red}$.

Towards this, note that $\deg(C'_{\red}) \le \deg(C)$, and so $\deg(C'_{\red})\leq \deg(A)d^{n-1}$. Applying Lemma~\ref{lem:genus-reg-bound} and Exercise IV.1.8 in \cite{hartshorne} to $C'_{\red}$, we see that
\[
p_{a}\left(\tilde{C}'_{\red}\right)\leq p_a\left(C'_{\red}\right)\leq \deg(A)^2d^{2n-2}+\deg(A)d^{n-1}-2.
\]
Since $\tilde{C}'_{\red}$ is an irreducible smooth curve, its geometric genus is equal to its arithmetic genus, and so
\[
g\left(\tilde{C}'_{\red}\right)=p_{a}\left(\tilde{C}'_{\red}\right)\leq\deg(A)^2d^{2n-2}+\deg(A)d^{n-1}-2.
\]
\end{proof}


\section{Application}\label{sec:applications}

As an application of Theorem~A, we show the existence of abelian varieties of the form $E^{n}\times A$ for  $n \in \mathbb{N}$, where $E$ is an elliptic curve, in the Torelli locus. Recall the Torelli locus $\cT_{g}$ is the image of the Torelli map
\[
\begin{tikzcd}[column sep = 3em, row sep = .75em]
\cM_{g} \rar & \cA_{g} \\
 C\rar[mapsto] & \Jac(C)
\end{tikzcd}
\]
between the moduli space of (geometrically irreducible, complete, smooth) curves of genus $g$ and the moduli space of principally polarized abelian varieties of dimension $g$.

Since the dimension of $\cM_{g}$ is $3g-3$ and the dimension of $\cA_g$ is $g(g+1)/2$, the Torelli locus is a proper subscheme of $\cA_g$ for $g\geq4$. In general describing this locus is hard, and relatively little is known. For example, given a principally polarized abelian variety of dimension greater than or equal to  $4$ over a finite field, it is difficult to determine whether it can be realized as the Jacobian variety of a smooth curve.

Further, since the codimension of $\cT_{g}$ grows with $g$, for any given stratification of $\mathcal{A}_g$, we expect the Torelli locus to only intersect the relatively generic strata. For example, if we fix an elliptic curve $E$ over $\fF_{q}$ then we may stratify $\cA_{g}$ by the number of copies of $E$ each abelian variety has as isogeny factors. That is to say each stratum has the form $\{E^{n}\}\times \cA_{g-n}$ for $0\leq n\leq g$. Then we expect the intersection $\cT_{g}\cap \{E^{n}\}\times \cA_{g-n}$ to often be empty for larger $n$. In particular, we expect the Jacobian of some smooth genus $g$ curve over $\fF_{q}$ to have $E^{n}$ as an isogeny factor only if $n$ is small relative to $g$. This is supported by the results in \cite{EHR}.

\begin{prop}\cite[Corollary 1.3]{EHR}
Let $E$ be an elliptic curve over a finite field $\fF_{q}$ of characteristic $p$ and $n\in \N$. Let $C$ be a smooth curve of genus $g$ defined over $\fF_{q}$. If $E^{n}$ is an isogeny factor of $\Jac(C)$ then
\[
g - \sqrt{\frac{\log \log g}{6 \log q}} \ge n.
\]
\end{prop}

\begin{proof}
By Corollary 1.3 in \cite{EHR}, $\Jac(C)$ has a simple factor $A$ with dimension at least $\sqrt{\frac{\log \log g}{6 \log q}}$. Thus, the dimension of the isogeny factor which decomposes as copies of $E$ is at most $g- \dim A$.
\end{proof}

The previous proposition can be viewed as a lower bound for the genus of curves with a prescribed isogeny factor for their Jacobians. Phrased differently, it says that for $g$ less than the explicit bound in the proposition, the intersection of $\{E^n\}\times \cA_{g-n}$ and $\cT_{g}$ is empty.  

On the other hand, our Theorem~A can be used to construct curves with a prescribed isogeny factor with bounded genus. In particular, Corollary~\ref{mainapp} implies that while unlikely, there does exist $g\leq B_{n,q}$ such that $\{E^n\}\times \cA_{g-n}$ intersects $\cT_{g}$. 

\begin{proof}[Proof of Corollary~\ref{mainapp}]
Let $E\subset\P^2$ be an elliptic curve defined over $\fF_q$ and consider the abelian variety $E^{n}$ with the polarization induced by divisor $E^{n-1} \times \{O\} + E^{n-2} \times \{O\} \times E + \ldots + \{O\} \times E^{n-1}$ which gives an embedding $E^n \subset \P^r$. By Theorem~A there exists a smooth geometrically connected curve $C$ defined over $\fF_q$ whose genus is explicitly bounded in terms of $n$, $\deg(E^n)$, $\zeta_{E^{n}}(n+\tfrac{1}{2})$, $q$ and $r$
such that the Jacobian of $C$ maps surjectively onto $E^n$. The surjectivity of the map $\Jac(C)\rightarrow E^{n}$ implies that $\Jac(C)$ admits a factor isogenous to $E^{n}$, and thus, it is enough show we can remove the dependence on $E$ from the genus on $C$, i.e. bound the terms $\deg(E^n)$, $\zeta_{E^{n}}(n+\frac{1}{2})$ and the dimension of the ambient projective space in terms of just $n$ and $q$.

Note that since $E$ is embedded in $\P^2$ with degree $3$, using the Segre embedding, $E^{n}$ is embedded in $\P^{3^n-1}$ with degree $3^n n!$. Further, using Weil Conjectures \cite{milneAV}*{Corollary~II.1.5} one can show that 
\[
\zeta_{E^{n}}\left(n+\tfrac{1}{2}\right)\leq \left(\tfrac{1+\sqrt{q}}{1-\tfrac{1}{\sqrt{q}}}\right)^{2^{n-1}}.
\]
With this, the genus bound for $C$ given by Theorem~A may be re-written independent of $E$. We may thus take $B_{n,q}$ to equal the resulting bound. 

\end{proof}

\begin{remark}
Recall the $a$ number of an abelian variety $A$ over a field $\kk$ of characteristic $p>0$ is defined as $ \dim_{\bar{\kk}} \Hom(\alpha_p, A[p])$ where $\alpha_{p}=\Spec \kk[x]/\langle x^p\rangle$. The previous corollary allows one to show the existence of Jacobian varieties over $\fF_{q}$ of bounded dimension with an $a$ number at least $n$. 
	
In particular, if in Corollary~\ref{mainapp} we take $E$ to be a supersingular elliptic curve, then with $C$ as in the corollary the $a$ number of $\Jac(C)$ is at least $n$. Previous results in this direction, see \cite{Pries}, mainly come from constructing special families of curves over $\overline{\fF}_p$, thus only provide existence over algebraically closed fields.
\end{remark}


\end{document}